\documentclass[11pt,a4paper,reqno,dvipsnames]{amsart}
\usepackage{geometry}
\usepackage{amsfonts}
\usepackage{amsmath}
\usepackage{amssymb,stmaryrd}
\usepackage{amsthm}
\usepackage{amsxtra}
\usepackage{bbm}
\usepackage{braket}
\usepackage{comment}
\usepackage{extarrows}
\usepackage{graphicx}
\usepackage{latexsym}
\usepackage{mathrsfs}
\usepackage{yhmath}
\usepackage{nicematrix}
\usepackage{mathtools}

\usepackage{float}
\usepackage{xcolor}
\usepackage{verbatim}
\usepackage{listings}
\usepackage[normalem]{ulem}

\usepackage{tikz}
\usetikzlibrary {patterns,patterns.meta} 
\usetikzlibrary{decorations.pathreplacing, decorations.markings}
\tikzset{->-/.style={decoration={markings,mark=at position #1 with {\arrow{>}}},postaction={decorate}}}

\usetikzlibrary{arrows}
\definecolor{qqwuqq}{rgb}{0,0.39215686274509803,0}
\definecolor{qqqqff}{rgb}{0,0,1}
\definecolor{uuuuuu}{rgb}{0.26666666666666666,0.26666666666666666,0.26666666666666666}
\definecolor{xdxdff}{rgb}{0.49019607843137253,0.49019607843137253,1}
\definecolor{ududff}{rgb}{0.30196078431372547,0.30196078431372547,1}
\definecolor{zzttqq}{rgb}{0.8,0,0.4}
\definecolor{qqwuqq}{rgb}{0,0.39215686274509803,0}
\definecolor{qqqqff}{rgb}{0,0,1}
\definecolor{uuuuuu}{rgb}{0.26666666666666666,0.26666666666666666,0.26666666666666666}
\definecolor{xdxdff}{rgb}{0.49019607843137253,0.49019607843137253,1}
\definecolor{ududff}{rgb}{0.30196078431372547,0.30196078431372547,1}
\usepackage{pgfplots}
\pgfplotsset{compat=1.15}

\usepackage[pdfborder={0 0 0}]{hyperref} % hyperref must be loaded after natbib
\hypersetup{colorlinks, citecolor=blue, linkcolor=blue, urlcolor=ForestGreen}
\hypersetup{pagebackref}
\usepackage{bm}

\usepackage{tikz}

\usepackage[initials,lite]{amsrefs} % amsrefs must be loaded after hyperref
\AtBeginDocument{%
    \def\MR#1{}
}
\usepackage{cleveref}
\Crefname{Lemma}{Lemma}{Lemmas}
\Crefname{Theorem}{Theorem}{Theorems}

\makeatother

\theoremstyle{plain}

\newtheorem{Theorem}{Theorem}[section]
\newtheorem{Lemma}[Theorem]{Lemma}
\newtheorem{Corollary}[Theorem]{Corollary}
\newtheorem{Proposition}[Theorem]{Proposition}

\newtheorem{Algorithm}[Theorem]{Algorithm}

\theoremstyle{definition}

\newtheorem{Assumptions and Discussion}[Theorem]{Assumptions and Discussion}

\newtheorem{Example}[Theorem]{Example}
\newtheorem{Definition}[Theorem]{Definition}

\newtheorem{Remark}[Theorem]{Remark}

\newtheorem{Observation}[Theorem]{Observation}

\newtheorem{Notation}[Theorem]{Notation}

\theoremstyle{remark}

\newtheorem*{acknowledgment*}{Acknowledgment}

\usepackage[shortlabels]{enumitem}
\SetEnumerateShortLabel{a}{\textup{(\alph*)}}
\SetEnumerateShortLabel{A}{\textup{(\Alph*)}}
\SetEnumerateShortLabel{1}{\textup{(\arabic*)}}
\SetEnumerateShortLabel{i}{\textup{(\roman*)}}
\SetEnumerateShortLabel{I}{\textup{(\Roman*)}}

\def\deg{\operatorname{deg}}
\def\ini{\operatorname{in}} % initial ideal/term

\def\ker{\operatorname{ker}}
\def\KK{{\mathbb K}}
\def\NN{{\mathbb N}}

\def\part{\operatorname{part}}

\def\supp{\operatorname{supp}}

\def\ZZ{{\mathbb Z}}

\newcommand\bda{{\bm a}}

\newcommand\bdB{{\bm B}}
\newcommand\bdb{{\bm b}}

\newcommand\bdc{{\bm c}}

\newcommand\bdd{{\bm d}}

\newcommand\bdi{{\bm i}}

\newcommand\bdj{{\bm j}}

\newcommand\bdp{{\bm p}}

\newcommand\bdq{{\bm q}}

\newcommand\bdT{{\bm T}}

\newcommand\bdx{{\bm x}}

\newcommand\bfA{\mathbf{A}}

\newcommand\bfB{\mathbf{B}}

\newcommand\bfC{\mathbf{C}}

\newcommand\bfI{\mathbf{I}}

\newcommand\calD{\mathcal{D}}

\newcommand\calF{\mathcal{F}}
\newcommand\calG{\mathcal{G}}
\newcommand\calH{\mathcal{H}}

\newcommand\calJ{\mathcal{J}}
\newcommand\calK{\mathcal{K}}

\newcommand\calR{\mathcal{R}}

\newcommand\Supp{\operatorname{Supp}}

\begin{document}

\title{Blowup Algebras of $n$--dimensional Ferrers Diagrams}

\author[Kuei-Nuan Lin, Yi-Huang Shen]{Kuei-Nuan Lin and Yi-Huang Shen}

\thanks{\today}

\thanks{
 %ring), analytic spread and related topics
    13F65  	%Commutative rings defined by binomial ideals, toric rings
    05E40 %Combinatorial aspects of commutative algebra
    Secondary 14M25  	%Toric varieties, Newton polyhedra, Okounkov bodies
}


\address{Department of Mathematics, The Penn State University, McKeesport, PA,
15132, USA}
\email{linkn@psu.edu}

\address{School of Mathematical Sciences, University of Science and Technology of China, Hefei, Anhui, 230026, P.R.~China}
\email{yhshen@ustc.edu.cn}

\begin{abstract}
    We demonstrate that the direct sum of ideals satisfying the strong $\ell$-exchange property is of fiber type.  Furthermore, we provide Gr\"obner bases of the presentation ideals of multi-Rees algebras and the corresponding special fibers, when they are associated with an $n$-dimensional Ferrers diagram that is standardizable. In particular, we show that these blowup algebras are
    Koszul Cohen--Macaulay normal domains and classify their singularities.
\end{abstract}

\maketitle

\section{Introduction}

Rees algebras play a central and pivotal role in the realms of commutative algebra and algebraic geometry, as extensively discussed in \cite{vasconcelos1994arithmetic}. They have a wide range of applications in various fields such as elimination theory, geometric modeling, and chemical reaction networks, as explained in \cites{CWL, Cox, CLS}.

Generalizing the classical investigation of the Rees algebra and the special fiber ring of a single ideal, we study the multi-Rees algebra associated with a collection of ideals.
Let $R=\KK[x_1,\ldots,x_n]$ be a polynomial ring over a field $\KK$, and
$ I_1, \ldots, I_r$ be homogeneous ideals in $R$.
The \emph{multi-Rees algebra} of this collection of ideals encapsulates the homogeneous coordinate ring of the blowup along the subschemes defined by these ideals. It is also identified as the Rees algebra of the module $I_1 \oplus \cdots \oplus I_r$, defined as the following multi-graded $R$-algebra:
\[
    \mathcal{R} (I_1 \oplus \cdots \oplus I_r) \coloneqq \bigoplus_{(a_1, \ldots, a_r) \in \mathbb{Z}_{\geq 0}^n} I_1^{a_1} \cdots I_r^{a_r}t_1^{a_1}\cdots t_r^{a_r} \subseteq R[t_1,\ldots, t_r].
\]
Here, auxiliary variables $t_1, \ldots, t_r$ are introduced. A related concept is the \emph{special fiber ring}, denoted by $\mathcal{F}(I_1 \oplus \cdots \oplus I_r)$. It corresponds to the image resulting from the blowup map, and can be expressed by $\mathcal{R}(I_1 \oplus \cdots \oplus I_r) \otimes_{R}\mathbb{K}$.
An important approach to studying the multi-Rees algebra and its special fiber is to represent them as quotients of polynomial rings respectively:
\[
    \mathcal{R} (I_1 \oplus \cdots \oplus I_r) \cong S[T_1,\ldots, T_\mu]/\mathcal{J}
\]
and
\[
    \mathcal{F} (I_1 \oplus \cdots \oplus I_r) \cong \mathbb{K}[T_1,\ldots, T_\mu]/\mathcal{K},
\]
where $\mu$ is the total number of minimal generators of $I_1, \ldots, I_r$.

A fundamental problem in the study of multi-Rees algebras is to find the implicit equations of the presentation ideals $\mathcal{J}$ and $\mathcal{K}$. Cox, Lin, and Sosa demonstrated that an explicit description of $\mathcal{J}$ and $\mathcal{K}$ provides insight into the equilibrium solutions of a chemical reaction network, as discussed in \cite{CLS}. 
Furthermore, when each of the ideals $I_1, \ldots, I_r$ is generated by monomials of the same degree, the presentation ideals $\mathcal{J}$ and $\mathcal{K}$ are generated by multigraded binomials.
Consequently, the isomorphisms $\mathcal{R} (I_1 \oplus \cdots \oplus I_r) \cong R[T_{1},\ldots, T_{\mu}]/\mathcal{J}$ and $\mathcal{F} (I_1 \oplus \cdots \oplus I_r) \cong \mathbb{K}[T_{1},\ldots, T_{\mu}]/\mathcal{K}$ yield coordinate rings of toric varieties, given that $\mathcal{J}$ and $\mathcal{K}$ are prime ideals.
Nevertheless, this implicit-equation problem remains open in general.

In addition to determining the presentation equations of the multi-Rees algebras, our interest lies in identifying instances where these algebras exhibit Koszul properties. A graded ring over a field $\mathbb{K}$ is \emph{Koszul} if the residue field $\mathbb{K}$ possesses a linear resolution over this ring. Koszul algebras are well-studied in classical commutative algebra and algebraic geometry, due to their favorable homological characteristics.
For example, it has been established that if $\mathcal{R}(I)$ is Koszul, then linear powers exist for $I$, meaning that every power of $I$ has a linear resolution \cite[Corollary 3.6]{Blum}. Bruns and Conca extended this result to the multigraded setting in \cite[Theorem 3.4]{BCresPowers}, showing that when $\mathcal{R} (I_1 \oplus \cdots \oplus I_r)$ is Koszul, products of $I_1,\ldots, I_r$ exhibit linear resolutions.

To establish the Koszul property of an algebra $A$, an effective strategy is to demonstrate its \emph{$G$-quadratic} nature. This involves expressing $A$ via an isomorphism $A \cong S/Q$, where $S$ is a polynomial ring and $Q \subset S$ is an ideal possessing a quadratic Gr\"obner basis with respect to a certain monomial order (refer to \cite[Theorem 6.7]{EH} or \cite{Froberg}).  Using this method, Lin and Polini established the Koszul property for multi-Rees algebras of powers of maximal ideals \cite[Theorem 2.4]{LinPolini}. DiPasquale and Jabbar Nezhad confirmed the validity of Bruns and Conca's conjecture for principal strongly stable ideals \cite[Corollary 6.4]{dipasquale2020koszul}.  Later, Kara, Lin, and Sosa demonstrated the Koszul nature of the multi-Rees algebra associated with two-generated strongly stable ideals \cite[Theorems 3.2 and 6.11]{KLS}.

In this work, we begin by extending the findings of Herzog, Hibi, and Vladoiu \cite[Theorem 5.1]{MR2195995} to the realm of multi-Rees algebras. Specifically, we introduce the concept of the \emph{strong $\ell$-exchange property} for a collection of ideals, and demonstrate that the direct sum of those ideals is of fiber-type. This means that the presentation ideal of the multi-Rees algebras is the sum of linear relationships and the presentation ideal of the fiber ring (\Cref{thm:ReesIdeal}). This result also generalizes 
\cite[Theorem 3.2]{KLS}
of Kara, Lin and Sosa, since any collection of strongly stable ideals satisfies the strong $\ell$-exchange property.

Our subsequent focus narrows down to a monomial ideal, $I_{\calD}$, associated with an $n$-dimensional Ferrers diagram, $\calD$, chosen for their relevance to statistics and coding theory (\cites{CNPY, GR}). Specifically, exploring the presentation ideals of their blowup algebras offers a solution for Maximum Likelihood Estimation problems (\cite{HKS}). Previous  studies (\cite{CN,CNPY}) have addressed the presentation ideals of $\calR(I_{\calD})$ and $\calF(I_{\calD})$ for $n=2$, 
proving these algebras to be Koszul. However, for $n=3$, the minimal generating sets of the presentation ideals for $\calR(I_{\calD})$ and $\calF(I_{\calD})$ could involve cubic generators in general, making these algebras non-Koszul (\cite[Example 2.4]{Lin-Shen3D}).  Due to this reason, Lin and Shen introduced the \emph{projection property} for the three-dimensional Ferrers diagram, and established the Koszul property for $\calR(I_{\calD})$ and $\calF(I_{\calD})$ under this condition (refer to \Cref{thm:3D}).
However, the monomial order, naturally extended from $n = 3$ to $n = 4$, does not lead to a $G$-quadratic quality under the projection property assumption. 
This leads us to the notion of \emph{standardizable diagrams} (\Cref{def:SDP}), which is a generalization of the \emph{strong projection property} in \cite{Lin-Shen3DMulti} for investigating the special fiber rings of a pair of three-dimensional Ferrers diagrams. 

In the final section of our work, we prove that both the multi-Rees algebra and the special fiber are $G$-quadratic (\Cref{ReesnD,FiberD}) when the associated $n$-dimensional Ferrers diagram is \emph{standardizable}. Consequently, these algebras are normal Cohen--Macaulay domains, and their singularities are classified (\Cref{CMFiber} and \Cref{ReesnD}). It is worth noting that our proof strategy is notably concise and efficient compared to the $n=2$ case in \cite{CNPY} and the $n=3$ case in \cite{Lin-Shen3D}.

\section{Multi-Rees Algebras and its special fibers}

Throughout this work, we will use the following notation convention.

\begin{Notation}
    \label{notation:collection}
    \begin{enumerate}
        \item To avoid confusion, the set of positive integers is denoted by $\ZZ_+$, and the set of non-negative integers is denoted by $\NN$.
        \item If $p$ is a positive integer, we denote the set $\{1,2,\dots,p\}$ by $[p]$.
        \item Let $R = \KK[\bdx] = \KK[x_1,\ldots,x_{n}]$ be the polynomial ring over a field $\KK$, endowed with the lexicographic order defined by $x_1>\cdots >x_n$. Furthermore, let $r$ be a positive integer and $I_1, \dots, I_r$ be a collection of monomial ideals in $R$. For each $i$, suppose that $\{f_{i,1}, \dots,f_{i,\mu_i}\}$ is the minimal monomial generating set of $I_i$, such that $\deg (f_{i,j})=d_i$ for all $j \in[\mu_i]$ and $f_{i,1} >_{lex} f_{i,2} >_{lex} \dots >_{lex} f_{i,\mu_i}$.
        \item Additionally, we frequently consider the polynomial rings $S = R[\mathbf{T}]=R[T_{i,j}:i\in [r],j\in [\mu_i]]$ and $T= \KK[\bdT]= \KK[T_{i,j}:i\in [r],j\in [\mu_i]]$.
    \end{enumerate}
\end{Notation}

To describe the presentations of the \emph{symmetric algebra} $\mathcal{S}(I_1 \oplus \cdots \oplus I_r)$, the \emph{multi-Rees algebra} $\mathcal{R}(I_1 \oplus \cdots \oplus I_r)$, and the \emph{special fiber ring} $\mathcal{F}(I_1 \oplus \cdots \oplus I_r)$, we introduce the following homomorphisms:
\begin{align*}
    \rho&: S \longrightarrow \mathcal{S}(I_1 \oplus \cdots \oplus I_r),\\
    \phi&: S \longrightarrow \mathcal{R}(I_1 \oplus \cdots \oplus I_r),\\
    \intertext{and}
    \psi&: T \longrightarrow \mathcal{F}(I_1 \oplus \cdots \oplus I_r).
\end{align*}
Here, for each $i$ and $j$, we require that $\rho(x_{i}) = x_{i}=\phi(x_{i})$, $\rho(T_{i,j})=f_{i,j}=\psi(T_{i,j})$, and $\phi(T_{i,j})=f_{i,j} t_i$.

\begin{Definition}
    \begin{enumerate}[a]
        \item Let $\mathcal{L}\coloneqq \ker(\rho)$, $\mathcal{J}\coloneqq \ker(\phi)$, and $\mathcal{K}\coloneqq \ker (\psi)$.
            They are known as the \emph{presentation ideals} of $\mathcal{S}(I_1 \oplus \cdots \oplus I_r)$, $\mathcal{R}(I_1 \oplus \cdots \oplus I_r)$, and $\mathcal{F}(I_1 \oplus \cdots \oplus I_r)$, respectively. Alternatively, they are sometimes referred to as the \emph{symmetric ideal}, the \emph{Rees ideal}, and the \emph{special fiber ideal}, respectively.
        \item 
            If $\mathcal{J}=\mathcal{L}+\mathcal{K}S$, then $I_1 \oplus \cdots \oplus I_r$ is \emph{of fiber type}.  
    \end{enumerate}
\end{Definition}

\begin{Remark}
    Since $\calR(I_1 \oplus \cdots \oplus I_r)$ is a subring of the integral domain $R[t_1,\dots,t_r]$, the Rees ideal $\calJ$ is prime. Similarly, $\calF(I_1 \oplus \cdots \oplus I_r)\cong \KK[f_{i,j}t_i:i\in [r],j\in[\mu_j]]$ is a semigroup ring. Hence, the special fiber ideal $\calK$ is also prime.
\end{Remark}

The notion of $\ell$-exchange property was introduced in \cite{MR2195995}. It is a valuable tool for demonstrating the fiber-type property of ideals that are close to strongly stable ideals, when considering their Rees algebras. In this work, we adapt this concept to handle the multi-Rees algebras in a similar manner.

\begin{Definition}
    \label{def:Strong_L_Ex_P}
    Let $I_1, \dots, I_r$ be the ideals in \Cref{notation:collection}.
    They are said to satisfy the \emph{strong $\ell$-exchange property}, if the
    following condition holds: for arbitrary $r$-tuple of non-negative integers
$(w_1,\dots,w_r)\in \NN^r$, and any monomials
    $u=\prod_{i=1}^r\prod_{j=1}^{w_i} f_{i,\alpha_{i,j}}$,
    $v=\prod_{i=1}^r \prod_{j=1}^{w_i} f_{i,\beta_{i,j}}$ satisfying
    \begin{enumerate}[i]
        \item $\deg_{x_t}(u)=\deg_{x_t}(v)$ for $t=1,\dots,q-1$ with $q\le n-1$,
        \item $\deg_{x_q}(u)<\deg_{x_q}(v)$,
    \end{enumerate}
    there exists a factor $f_{i,\alpha_{i,j}}$ of $u$, and an integer $q'$ with
    $q<q' \le n$ such that $x_qf_{i,\alpha_{i,j}}/x_{q'}\in I_i$.
\end{Definition}

The theorem presented below is a straightforward generalization of \cite[Theorem 5.1]{MR2195995} to the case of multi-Rees algebras.

\begin{Theorem}
    \label{thm:ReesIdeal}
    Assume that $I_1,\dots,I_r$ are ideals as in \Cref{notation:collection} which satisfy the strong $\ell$-exchange property.  Let $\calH$ be the collection of binomials $x_{k_1}T_{i,j}-x_{k_2}T_{i,j'}$ where $x_{k_1}f_{i,j}=x_{k_2}f_{i,j'}$, $k_1<k_2$, and $k_2$ is the largest for which $x_{k_1}f_{i,j}/x_{k_2}$ belongs to $I_i$. Suppose that $\mathcal{G}$ is a Gr\"obner basis for the special fiber ideal $\calK$ with respect to a monomial order $\tau$ on $T$. Then $\mathcal{G}' \coloneqq  \calH \cup \mathcal{G}$ is a Gr\"obner basis for the Rees ideal $\calJ$ with respect to an extended monomial order $\tau'$ on $S$ such that $\tau'|_{T}=\tau$.  In addition, if $\calG$ is a reduced Gr\"obner basis, then so is $\calG'$.
\end{Theorem}

\begin{proof}
    Given two monomials $\bdx^\bda \bdT^\bdb$ and $\bdx^\bdc \bdT^\bdd$ in $S$, we set $\bdx^\bda \bdT^\bdb >_{\tau'} \bdx^\bdc \bdT^\bdd$ if and only if (i) $\bdx^\bda  >_{lex} \bdx^\bdc$ or (ii) $\bdx^\bda = \bdx^\bdc$ and $\bdT^\bdb >_{\tau} \bdT^\bdd$. It is the product order on $S$ of the lexicographic order on $R$ and the order $\tau$ on $T$ in the sense of \cite[page 17]{EH}.

    Take any $g\in \calJ$. Since the prime ideal $\calJ$ has a binomial generating set, 
    we can assume that $g$ is an irreducible binomial. We have two cases.
    \begin{enumerate}[a]
        \item Suppose that $g\in \calK$. Therefore, $\ini_{\tau'}(g)=\ini_{\tau}(g)$ is divisible by $\ini_{\tau'}(h)=\ini_{\tau}(h)$ for some $h\in \calG$.
        \item Suppose that $g\notin \calK$. Therefore, we may assume that $g=x_q\bdx^{\bda}\bdT^{\bdb} - x_{q'}\bdx^{\bda'}\bdT^{\bdb'}$ where $q$ is the smallest such that $x_q$ appears in the binomial $g$. Since $g$ is irreducible, $q\ne q'$ and $q\notin \supp(\bdx^{\bda'})$. Thus, $\deg_{x_q}(u)<\deg_{x_q}(v)$, where $u\coloneqq \psi(\bdT^{\bdb})$ and $v\coloneqq \psi(\bdT^{\bdb'})$. Whence, by the strong $\ell$-exchange property of $I_1,\dots,I_r$, we can find $T_{i,\alpha_{i,j}}$ dividing $\bdT^{\bdb}$ and $q''>q$ such that $x_qf_{i,\alpha_{i,j}}/x_{q''}\in I_i$. Let $q''$ be the largest such that $x_qf_{i,\alpha_{i,j}}/x_{q''}\in I_i$ and write $f_{i,s}= x_qf_{i,\alpha_{i,j}}/x_{q''}$. Whence, $h\coloneqq x_qT_{i,\alpha_{i,j}}-x_{q''}T_{i,s}\in \calJ$ with $\ini_{\tau'}(h)=x_qT_{i,\alpha_{i,j}}$ dividing $\ini_{\tau'}(g)=x_q\bdx^{\bda}\bdT^{\bdb}$.
    \end{enumerate}
    By the above discussion, it is clear now that $\calG'$ is a Gr\"obner basis of $\calJ$ with respect to $\tau'$. Furthermore, if $\calG$ is a reduced Gr\"obner basis, then it is clear that $\calG'$ is also a reduced Gr\"obner basis.
\end{proof}

\begin{Remark}
    A simple adaption of the argument presented in \cite[Example 4.2]{MR2195995} shows that any set of equigenerated strongly stable ideals satisfies the strong $\ell$-exchange property. Thus, \Cref{thm:ReesIdeal} also offers a succinct proof of the recent finding in \cite[Theorem 3.12]{KLS} by Kara, Lin and Sosa.
\end{Remark}

\section{Ferrers diagrams}

The final goal of this paper is to discuss the multi-Rees algebra linked to an $n$-dimensional Ferrers diagram. To begin with, we introduce essential concepts related to this combinatorial object.

\begin{Definition}
    \label{nF}
    \begin{enumerate}[a]
        \item An \emph{$n$-dimensional Ferrers diagram} is a nonempty set $\calD$ of finite lattice points in $ \ZZ_+^n$ with the property that, if $\bdj=(j_1,\ldots,j_n) \in \calD$, and $\bdi=(i_1,\ldots,i_n)\in \ZZ_+^n$ with $i_k\le j_k$ for each $k\in [n]$, then $\bdi\in \calD$. We say $\calD$ is \emph{rectangular} if there exits a point $(a_1,\ldots,a_n) \in \calD$ such that
            \[
                \calD=\Set{(i_1,\dots,i_n):1\le i_k \le a_k, 1\le k\le n}.
            \]
        Whence, it is the Cartesian product $[a_1]\times [a_2] \times \cdots \times [a_n]$.

        \item Given an $n$-dimensional Ferrers diagram $\calD$, let $m$ be a fixed positive integer such that
            \[
                m\ge \max\Set{j_1,\dots,j_n:(j_1,\ldots, j_n)\in\calD}.
            \]
            We will consider the associated polynomial ring
            \[
                R=R_\calD\coloneqq \KK[x_{i,j} : i\in [n], j\in [m]].
            \]
            For each $\bda=(a_1,\dots,a_n)\in \calD$, let $\bdx_\bda\coloneqq \prod_{i=1}^n x_{i,a_i}\in R$. Then, we have the \emph{Ferrers ideal}
            \[
                I_{\calD}\coloneqq  \left( \bdx_{\bda}:\bda \in \calD \right)\subset R.
            \]
    \end{enumerate}
\end{Definition}

The candidates for the presentation equations of the blowup algebras of $n$-dimensional Ferrers diagrams are similar to those of the three-dimensional case.

\begin{Definition}
    \label{2-minors}
    Let $\calD$ be an $n$-dimensional Ferrers diagram.
    \begin{enumerate}[a]
        \item Suppose that $H$ is a subset of $[n]$. For every $\bda=(a_1,\ldots,a_{n}),\bdb=(b_1,\ldots,b_{n})$ in $\calD $, we define $\bdp =(p_1,\ldots,p_{n})$ and $\bdq =(q_1,\ldots,q_{n})$ with $p_i=a_i$ and $q_i=b_i$ if $i\notin H$, and $p_i=b_i$ and $q_i=a_i$ if $i\in H$. Then, we introduce the \emph{interchange} of $\bda$ and $\bdb$ with respect to $H$:
            \[
                \bfI_{H}(\bda,\bdb)\coloneqq
                \begin{cases}
                    T_{\bda}T_{\bdb}-T_{\bdp}T_{\bdq}, & \text{if
                    $\bdp, \bdq\in \calD $},\\
                    0, & \text{otherwise},
                \end{cases}
            \]
            which is considered as an element in $\KK[\bdT] =\KK[T_{\bda}:\bda\in \calD]$.  Notice that $\bda=\bdb$ if and only if $\bdp=\bdq$, which will imply that $\bfI_H(\bda,\bdb)=0$.
        \item Furthermore, let
            \[
                \bfI(\calD)\coloneqq \Set{\bfI_{H}(\bda,\bdb) : H\subseteq [n], \bda,\bdb\in \calD, \bfI_H(\bda,\bdb)\ne 0 }.
            \]
            This is a collection of binomials with squarefree terms.
    \end{enumerate}
\end{Definition}

In the next section, we are going to generalize the following result to higher dimensional case.

\begin{Theorem}
    [{\cite[Theorem 6.1 and 6.5]{Lin-Shen3D}}]
    \label{thm:3D}
    Let $\calD$ be a three-dimensional Ferrers diagram that satisfies the \emph{projection property} in the sense of \cite[Definition 2.5]{Lin-Shen3D}. The set $\bfI(\calD)$ is a quadratic Gr\"{o}bner basis of $\calK$ with respect to a lexicographic order and $\calF(I_{\calD})$ is a Koszul Cohen--Macaulay normal domain. Moreover, the Rees algebra $\calR(I_{\calD})$ is Koszul and the ideal $I_{\calD}$ is of fiber type.
\end{Theorem}

However, for higher dimensional case, the projection property condition as in \Cref{thm:3D} is not enough; see \Cref{exam:need_SPP} below.

\begin{Example}
    \label{exam:need_SPP}
    Let $\calD$ be the minimal four-dimensional Ferrers diagram containing the following lattice points:
    \[
        (1, 3, 4, 4), \;
        (1, 4, 3, 3), \;
        (2, 1, 4, 4), \;
        (2, 2, 3, 3), \;
        (3,1,3,3), \;
        (3,2,2,2).
    \]
    It is not difficult to verify that $\calD$ satisfies the projection property. However, the reduced Gr\"obner basis of the presentation ideal $\calK$ for the special fiber ring $\calF(I_\calD)$ contains $25$ additional binomials of degree $3$ when using the natural lexicographic order. 
\end{Example}

Therefore, we introduce the \emph{standardizability},
which can be seen as a generalization of the \emph{strong projection property} introduced in \cite{Lin-Shen3DMulti} for studying the special fiber rings of a pair of three-dimensional Ferrers diagrams. 
The rest of the paper is devoted to the computation of a Gr\"obner basis for the presentation ideal of the multi-Rees algebra associated with a standardizable Ferrers diagram.

\begin{Definition}
    \label{def:SDP}
    An $n$-dimensional Ferrers diagram $\calD$ is called \emph{standardizable} if for every $\bda \neq \bdb \in \calD$ such that $a_i=b_i$ for $i<k$ and $a_k<b_k$, there exists a lattice point
    \[
        \bdp=(a_1,\ldots,a_{k},\max\{a_{k+1},b_{k+1}\},\ldots, \max\{a_n,b_n\})\in \calD.
    \]
\end{Definition}

\begin{Example}
    \label{exam:SPP_diagram}
    \Cref{Shadow} displays 
    the smallest three-dimensional Ferrers diagram that contains the lattice points $(1,3,3)$, $(2,2,3)$, $(2,3,2)$, $(3,1,2)$, and $(3,2,1)$. It is also the smallest \emph{standardizable} Ferrers diagram that contains the lattice points $(2, 1, 3)$, $(2, 3, 1)$, $(3, 1, 2)$, and $(3, 2, 1)$.
    \begin{center}
        \begin{figure}[htbp] 
        \begin{tikzpicture}[every node/.style={scale=0.8}, line cap=round,line join=round,>=triangle 45,x=1cm,y=1cm, scale=0.50]
                \fill[pattern=crosshatch, fill opacity=0.5] (-2,-1) -- (0,-2) -- (0,-4) -- (-2,-5) -- (-4,-4) -- (-4,-2) -- cycle;
                \fill[pattern=north east lines,fill opacity=0.5]  (0,-2) -- (0,-4) -- (2,-5) -- (4,-4) -- (4,-2) -- (2,-3) -- cycle;
                \fill[pattern=crosshatch,fill opacity=0.5] (-4,2) -- (-2,1) -- (-2.0,-1) -- (-4,-2) -- (-6,-1) -- (-6,1) -- cycle;
                \fill[pattern=north east lines, fill opacity=0.5] (-4,2) -- (-2,1) -- (-2,3) -- (0,4) -- (-2,5) -- (-4,4) -- cycle;
                \draw[pattern={Lines[angle=135,distance=3pt]}] (4,4) -- (6,3) -- (6,1) -- (4,0) -- (2,1) -- (2,3) -- cycle;
                \draw[pattern={Lines[angle=135,distance=3pt]}] (0,4) -- (2,3) -- (2,1) -- (0,0) -- (-2,1) -- (-2,3) -- cycle;
                \draw[pattern={Lines[angle=135,distance=3pt]}] (2,1) -- (4,0) -- (4,-2) -- (2,-3) -- (0,-2) -- (0,0) -- cycle;

                \draw [line width=1.2pt] (-2,-1)-- (0,-2);
                \draw [line width=1.2pt] (-2,-3)-- (-2,-5);
                \draw [line width=1.2pt] (-2,-3)-- (0,-2);
                \draw [line width=1.2pt] (-2,1)-- (-2,-1);
                \draw [line width=1.2pt] (-2,1)-- (-4,0);
                \draw [line width=1.2pt] (-2,1)-- (0,0);
                \draw [line width=1.2pt] (-2,3)-- (-2,1);
                \draw [line width=1.2pt] (-2,3)-- (0,2);
                \draw [line width=1.2pt] (-2,5)-- (4,2);
                \draw [line width=1.2pt] (-4,-2)-- (-2,-1);
                \draw [line width=1.2pt] (-4,-2)-- (-2,-3);
                \draw [line width=1.2pt] (-4,-2)-- (-4,-4);
                \draw [line width=1.2pt] (-4,-4)-- (-2,-5);
                \draw [line width=1.2pt] (-4,0)-- (-4,-2);
                \draw [line width=1.2pt] (-4,2)-- (-2,1);
                \draw [line width=1.2pt] (-4,2)-- (-6.0,1.0);
                \draw [line width=1.2pt] (-4,4)-- (-2,3);
                \draw [line width=1.2pt] (-4,4)-- (-4,2);
                \draw [line width=1.2pt] (-6,-1)-- (-4,-2);
                \draw [line width=1.2pt] (-6,-3)-- (-4,-4);
                \draw [line width=1.2pt] (-6,1)-- (-4,0);
                \draw [line width=1.2pt] (-6.0,-1.0)-- (-6.0,-3.0);
                \draw [line width=1.2pt] (-6.0,1.0)-- (-6.0,-1.0);
                \draw [line width=1.2pt] (0,-2)-- (0,-4);
                \draw [line width=1.2pt] (0,-2)-- (2,-3);
                \draw [line width=1.2pt] (0,-4)-- (-2,-5);
                \draw [line width=1.2pt] (0,-4)-- (2,-5);
                \draw [line width=1.2pt] (0,0)-- (0,-2);
                \draw [line width=1.2pt] (0,0)-- (2.0,-1.0);
                \draw [line width=1.2pt] (0,2)-- (0,0);
                \draw [line width=1.2pt] (0,6)-- (-4,4);
                \draw [line width=1.2pt] (0,6)-- (6,3);
                \draw [line width=1.2pt] (2,3)-- (2.0,1.0);
                \draw [line width=1.2pt] (2,5)-- (-2,3);
                \draw [line width=1.2pt] (2.0,-1.0)-- (2.0,-5.0);
                \draw [line width=1.2pt] (2.0,-5.0)-- (4,-4);
                \draw [line width=1.2pt] (2.0,1.0)-- (0,0);
                \draw [line width=1.2pt] (2.0,1.0)-- (4,0);
                \draw [line width=1.2pt] (4,-2)-- (2.0,-3.0);
                \draw [line width=1.2pt] (4,-4)-- (6,-3);
                \draw [line width=1.2pt] (4,0)-- (2.0,-1.0);
                \draw [line width=1.2pt] (4,2)-- (4,-4);
                \draw [line width=1.2pt] (4,2)-- (6,3);
                \draw [line width=1.2pt] (4,4)-- (0,2);
                \draw [line width=1.2pt] (6,-1)-- (4,-2);
                \draw [line width=1.2pt] (6,1)-- (4,0);
                \draw [line width=1.2pt] (6,3)-- (6,-3);
              
                \draw (-3.1,4.5) node[anchor=north west] {\textbf{(2,1,3)}};
                \draw (-1.1,3.5) node[anchor=north west] {\textbf{(2,2,3)}};
                \draw (-3.1,-1.5) node[anchor=north west] {\textbf{(3,2,1)}};
                \draw (-5.1,1.5) node[anchor=north west] {\textbf{(3,1,2)}};
                \draw (0.9,0.5) node[anchor=north west] {\textbf{(2,3,2)}};
                \draw (1.9,-3) node[anchor=north west] {\textbf{(2,3,1)}};
                \draw (2.9,3.5) node[anchor=north west] {\textbf{(1,3,3)}};
                \begin{scriptsize}
                    \draw [fill=white] (-2,1) circle (2.5pt);
                    \draw [fill=white] (-2,3) circle (2.5pt);
                    \draw [fill=white] (-2,5) circle (2.5pt);
                    \draw [fill=white] (-2.0,-1.0) circle (2.5pt);
                    \draw [fill=white] (-2.0,-3.0) circle (2.5pt);
                    \draw [fill=white] (-2.0,-5.0) circle (2.5pt);
                    \draw [fill=white] (-4,-2) circle (2.5pt);
                    \draw [fill=white] (-4,-4) circle (2.5pt);
                    \draw [fill=white] (-4,0) circle (2.5pt);
                    \draw [fill=white] (-4,2) circle (2.5pt);
                    \draw [fill=white] (-4,4) circle (2.5pt);
                    \draw [fill=white] (-6.0,-1.0) circle (2.5pt);
                    \draw [fill=white] (-6.0,-3.0) circle (2.5pt);
                    \draw [fill=white] (-6.0,1.0) circle (2.5pt);
                    \draw [fill=white] (0,-2) circle (2.5pt);
                    \draw [fill=white] (0,-4) circle (2.5pt);
                    \draw [fill=white] (0,0) circle (2.5pt);
                    \draw [fill=white] (0,2) circle (2.5pt);
                    \draw [fill=white] (0,4) circle (2.5pt);
                    \draw [fill=white] (0,6) circle (2.5pt);
                    \draw [fill=white] (2,3) circle (2.5pt);
                    \draw [fill=white] (2,5) circle (2.5pt);
                    \draw [fill=white] (2.0,-1.0) circle (2.5pt);
                    \draw [fill=white] (2.0,-3.0) circle (2.5pt);
                    \draw [fill=white] (2.0,-5.0) circle (2.5pt);
                    \draw [fill=white] (2.0,1.0) circle (2.5pt);
                    \draw [fill=white] (4,-2) circle (2.5pt);
                    \draw [fill=white] (4,-4) circle (2.5pt);
                    \draw [fill=white] (4,0) circle (2.5pt);
                    \draw [fill=white] (4,2) circle (2.5pt);
                    \draw [fill=white] (4,4) circle (2.5pt);
                    \draw [fill=white] (6,-1) circle (2.5pt);
                    \draw [fill=white] (6,-3) circle (2.5pt);
                    \draw [fill=white] (6,1) circle (2.5pt);
                    \draw [fill=white] (6,3) circle (2.5pt);
                \end{scriptsize}
            \end{tikzpicture}
            \caption{A three-dimensional standardizable Ferrers diagram}
            \label{Shadow}
        \end{figure}
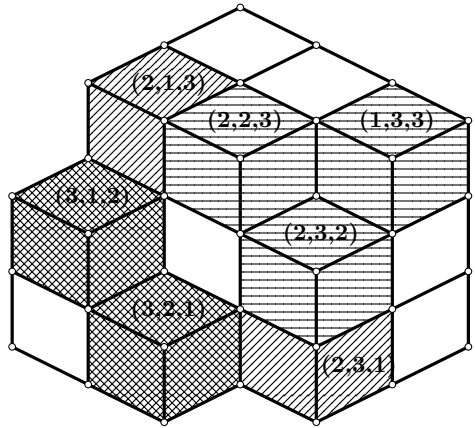
    \end{center}
\end{Example}

\begin{Remark}
    In general, there is no poset structure on the standardizable diagram $\calD$ such that $\bda,\bdb\le \bdp$ and $\bdp=\max(\bda,\bdb)$, using the notation in \Cref{def:SDP}. The reader can verify this with the lattice points $(3,1,1)$, $(3,2,1)$, and $(2,1,1)$ in \Cref{exam:SPP_diagram}.
\end{Remark}

\section{Ordering of tableau}

The aim of this paper is to extend \Cref{thm:3D} to multi-Rees algebras of arbitrary dimensional Ferrers diagrams. 
In this section, we introduce a useful total order on $\ZZ_+^{n}$, inspired by the work of De Negri \cite{DeNegri}. We will use this order in the following section when examining the presentation ideals of the blowup algebras.

\begin{Definition}
    \label{orderTuple}
    Let $n$ be a positive integer.
    \begin{enumerate}[a]
        \item For vectors $\bda,\bdb \in \ZZ_+^{n}$, we say $\bda>_{\sigma}\bdb$ if the first nonzero entry of $\bda-\bdb$ is negative. We will denote this total order by $\sigma$.

        \item In this paper, by \emph{tableau} we mean a matrix with entries in $\ZZ_+$. The tableau can be visualized as a rectangular collection of boxes filled with positive integers, similar to the \emph{Young tableau} in the literature. If $\bfA=[a_{i,j}]$ is a $p\times n$ tableau with row vectors $\bda^1,\dots,\bda^p$, it is written as $\bfA=[\bda^1,\dots,\bda^p]$ by abuse of notation. It is called \emph{semi-standard} if the row vectors satisfy
            \begin{equation}
                \bda^{1} \ge_{\sigma} \bda^{2} \ge_{\sigma} \dots \ge_{\sigma} \bda^{p}.
                \tag{Cond-1}
                \label{eqn:cond_1}
            \end{equation}
        \item Given two $p\times n$ semi-standard tableaux $\bfA=[\bda^1,\ldots, \bda^p]$ and $\bfB=[\bdb^1,\ldots,\bdb^p]$, we say $\bfA>_{\sigma}\bfB$ if there is an integer $k\in [p]$ such that $\bda^i=\bdb^i$ for $i\in [k-1]$ and $\bda^{k}>_{\sigma}\bdb^k$.
        \item Given a $p\times n$ tableau $\bfA=[a_{i,j}]$, for each $j\in [n]$, we will denote the multiset $\{a_{1,j},a_{2,j},\dots,a_{p,j}\}$ by $\supp_j(\mathbf{A})$. Suppose that $\bfB$ is another tableau of the same size. We say that $\supp(\bfA)=\supp(\bfB)$ if  $\supp_j(\bfA)=\supp_j(\bfB)$ for all $j\in [n]$. 
        \item A semi-standard tableaux $\bfA$ is called \emph{standard} if for every semi-standard tableau $\bfB$ with $\supp(\bfA)=\supp(\bfB)$ and $\bfA\neq \bfB$, we have $\bfB>_{\sigma} \bfA$.
    \end{enumerate}
\end{Definition}

\begin{Remark}
    \label{rmk:reductions}
    \begin{enumerate}[a]
        \item If $\bfA$ is a tableau, we can turn it into a semi-standard tableau by swapping the rows. 
        \item \label{rmk:reductions_b}
        Any subtableau obtained from a semi-standard tableau by removing rows or rightmost columns remains semi-standard.
        \item If $\bfA$ is a semi-standard tableau, the set
            \[
                \Set{\bfB : \text{$\bfB$ is semi-standard with $\supp(\bfB)=\supp(\bfA)$}}
            \]
            is finite. Since $\sigma$ is a total order on this set, there exists a unique standard tableau $\bfB$ in it.
    \end{enumerate} 
\end{Remark}

Next, we present an algorithm for finding the standard tableau from a given tableau. It is based on the following simple fact:

\begin{Observation}
    \label{obs_inc}
    Given a chain of elements $\bda^1 \succeq \bda^2 \succeq \cdots \succeq \bda^p$ in some poset and a multiset of integers $B=\{b_1,\dots,b_p\}$, the elements in $B$ can be uniquely sorted so that the following conditions holds:
    \begin{itemize}
        \item if $\bda^i \succ \bda^j$, then $b_i\ge b_j$, and
        \item if $\bda^i=\bda^j$ with $i<j$, then $b_i\le b_j$.
    \end{itemize}
\end{Observation}

\begin{Algorithm}
    \label{algo_standard_tab}
    Given a $p\times n$ semi-standard tableau $\bfA=[a_{i,j}]$, we use the following method to find a $p\times n$ standard tableaux $\bfB=[b_{i,j}]$ such that $\Supp(\bfA)=\Supp(\bfB)$.
    \begin{enumerate}[1]
        \item
            Firstly, since $\bfA$ is semi-standard, we have $(a_{1,1})\ge_{\sigma}(a_{2,1})\ge_{\sigma}\cdots \ge_{\sigma}(a_{p,1})$, i.e., $a_{1,1}\le a_{2,1} \le \cdots \le a_{p,1}$. Whence, we set $b_{i,1}=a_{i,1}$ for $i\in [p]$.
        \item Secondly, for $j=1,2,\dots,n$, suppose that the first $j$ columns of $\bfB$ have been determined. For $1\le i\le p$, let $\bdb^i_{j}$ be the vector $(b_{i,1},b_{i,2},\dots,b_{i,j})\in \ZZ_+^j$. Whence, we have $\bdb^1_{j} \ge_{\sigma} \bdb^2_{j} \ge_{\sigma}\cdots \ge_{\sigma} \bdb^p_{j}$.
        \Cref{obs_inc} implies that the multiset $\{a_{1,j+1},a_{2,j+1},\dots,a_{p,j+1}\}$ can be rearranged to $\{b_{1,j+1},b_{2,j+1},\dots,b_{p,j+1}\}$ such that
            \begin{itemize}
                \item if $\bdb^i_{j}>_{\sigma} \bdb^{i'}_{j}$, then $b_{i,j+1}\ge b_{i',j+1}$, and \hfill \textup{(Cond-2)}
                \item if $\bdb^i_{j}=\bdb^{i'}_{j}$ with $i<i'$, then $b_{i,j+1} \le b_{i',j+1}$. \hfill \textup{(Cond-3)}
            \end{itemize}
            Consequently, the $(j+1)$-th column of $\bdB$ 
            is determined. Furthermore, we have vectors $\bdb^1_{j+1},\dots,\bdb^p_{j+1}$, where each $\bdb^i_{j+1}=(b_{i,1},\dots,b_{i,j},b_{i,j+1})\in \ZZ_+^{j+1}$.
            It is clear from \textup{(Cond-2)} and \textup{(Cond-3)} that $\bdb^1_{j+1} \ge_{\sigma} \bdb^2_{j+1} \ge_{\sigma} \cdots \ge_{\sigma} \bdb^p_{j+1}$.
    \end{enumerate}
    After all these computations, we get a semi-standard tableau $\bfB$ with $\supp(\bfA)=\supp(\bfB)$.
\end{Algorithm}

\begin{Example}
    The $12 \times 5$ tableau $\bfA$ in \Cref{tab:standard} is semi-standard, but not standard. We can apply \Cref{algo_standard_tab} to get a standard tableau $\bfB$ with the same support.
    \begin{figure}[h]
        $\bfA=$
        \begin{tabular}{|c|c|c|c|c|}
            \hline
            1&1&1&1&1\\ \hline
            1&1&1&2&1\\ \hline
            1&2&2&2&2\\ \hline
            1&2&2&2&3\\ \hline
            2&2&2&3&4\\ \hline
            2&2&2&4&5\\ \hline
            2&3&3&4&6\\ \hline
            2&3&3&5&6\\ \hline
            2&3&4&6&6\\ \hline
            3&3&5&6&7\\ \hline
            3&4&5&7&7\\ \hline
            3&4&6&7&8\\ \hline
        \end{tabular}$ \qquad \xLongrightarrow{\text{\Cref{algo_standard_tab}}} \qquad \bfB=$
        \begin{tabular}{|c|c|c|c|c|}
            \hline
            1 & 3 & 5 & 7 & 8\\
            \hline
            1 & 3 & 6 & 7 & 7\\
            \hline
            1 & 4 & 4 & 6 & 7\\
            \hline
            1 & 4 & 5 & 6 & 6\\
            \hline
            2 & 2 & 2 & 5 & 6\\
            \hline
            2 & 2 & 3 & 4 & 5\\
            \hline
            2 & 2 & 3 & 4 & 6\\
            \hline
            2 & 3 & 2 & 2 & 4\\
            \hline
            2 & 3 & 2 & 3 & 3\\
            \hline
            3 & 1 & 1 & 2 & 2 \\
            \hline
            3 & 1 & 2 & 2 & 1 \\
            \hline
            3 & 2 & 1 & 1 & 1\\
            \hline
        \end{tabular}
        \caption{Two $12 \times 5$ semi-standard tableaux with the same support}
        \label{tab:standard}     
    \end{figure}
    Notice that, this $\bfB$ is not sorted in the sense of Sturmfels \cite{Sturmfels}, nor standard in the sense of De Negri \cite{DeNegri}.
\end{Example}

Applying \Cref{algo_standard_tab} yields standard tableaux, which are guaranteed to be valid by the following result. 

\begin{Lemma}
    \label{lem:algo_works}
    The tableau $\bfB$ constructed in \Cref{algo_standard_tab} is standard.
\end{Lemma}

\begin{proof}
    Let $\bfC=[c_{i,j}]$ be an arbitrary semi-standard tableau with $\supp(\bfC)=\supp(\bfB)$. If $\bfC\ne \bfB$, we will prove that $\bfC>_{\sigma} \bfB$. Therefore, $\bfB$ is standard.

    Let $j_0=\min\{j: c_{i,j}\ne b_{i,j} \text{ for some $i\in [p]$ and $j\in [n]$}\}$ and $i_0=\min\{i\in [p]:c_{i,j_0}\ne b_{i,j_0}\}$.
    Based on \Cref{rmk:reductions}\ref{rmk:reductions_b}, we can assume that $j_0=n$ and $i_0=1$.

    For each $i$ and $j$, let $\bdc^i_{j}$ be the vector $(c_{i,1},c_{i,2},\dots,c_{i,j})$. Whence, $\bdc^i_{j}=\bdb^i_{j}$ for $i\in [p]$ and $j\in [n-1]$ by the choice of $i_0$ and $j_0$. Let $i_1 \coloneqq \max\{i: \bdb^i_{n-1}=\bdb^1_{n-1}\}$. Consequently, $\bdb^1_{n-1}=\bdb^2_{n-1}=\cdots=\bdb^{i_1}_{n-1}$. 
It follows from \eqref{eqn:cond_1} that $c_{1,n}\le c_{2,n} \le \cdots \le c_{i_1,n}$ and  $b_{1,n}\le b_{2,n}\le \cdots \le b_{i_1,n}$.
Furthermore, (Cond-2) implies that the multiset $\{b_{1,n},b_{2,n},\dots,b_{i_1,n}\}$ is the $i_1$-subset of $\supp_{n}(\bfA)$ with the largest sum. Since $c_{1,n}\ne b_{1,n}$ by the choice of $i_0$ and $j_0$, we have $\{b_{1,n},b_{2,n},\dots,b_{i_1,n}\}\ne \{c_{1,n},c_{2,n},\dots,c_{i_1,n}\}$. This implies that $[\bdc^1_{n}, \ldots,\bdc^{i_1}_{n}] >_{\sigma} [\bdb^{1}_{n},\ldots,\bdb^{i_1}_{n}]$. Since $\bdc^{i}_{j}=\bdb^i_{j}$ for $i\in [p]$ and $j\in [n-1]$, and
    \[
        \bdb^1_{n-1}=\bdb^2_{n-1}=\cdots=\bdb^{i_1}_{n-1}>_{\sigma}\bdb^{i_1+1}_{n-1}\ge_{\sigma} \cdots \ge_{\sigma} \bdb^{p}_{n-1},
    \]
    this implies that $[\bdc^1_{n},\ldots,\bdc^{p}_{n}]>_{\sigma} [ \bdb^{1}_{n}, \ldots, \bdb_{n}^p]$, i.e., $\bfC>_{\sigma} \bfB$. 
\end{proof}

Next, we present a useful tool for verifying the 
standardness
of a given semi-standard tableau.

\begin{Lemma}
    \label{StandardProp}
    Let $\bda$ and $\bdb$ be two tuples in $\ZZ_+^{n}$ such that $\bda\ge_\sigma \bdb$. Then, the semi-standard tableau $[\bda,\bdb]$ is standard if and only if $\bda=\bdb$, or there exists a $k\in [n]$ such that $a_{i}=b_{i}$ for all $i<k$, and $a_k<b_{k}$, but $a_{j}\geq b_j$ for all $k<j\leq n$.
\end{Lemma}

\begin{proof}
    If $\bda=\bdb$, then for any semi-standard tableau $[\bdp,\bdq]$ with $\supp ([\bdp,\bdq])=\supp([\bda,\bdb])$, we must have $\bdp=\bdq=\bda=\bdb$. Therefore, $[\bda,\bdb]$ is standard. Thus, in the following, we will assume $\bda\neq \bdb$.

    Suppose that there exists a value $k$ such that $a_{i}=b_{i}$ for all $i<k$, and $a_k<b_{k}$ but $  a_{j}\geq b_j$ for all $k<j\leq n$. Then, for any semi-standard tableau $[\bdp,\bdq]$ with $\supp ([\bdp,\bdq])=\supp([\bda,\bdb])$, we have $p_i=a_{i}=b_{i}=q_i$ for all $i<k$. It is clear that $a_k=\min\{a_k,b_k\}=\min\{p_k,q_k\}=p_k$ by \eqref{eqn:cond_1}. Additionally, $a_j = \max\{a_j,b_j\} =\max\{p_j,q_j\}\ge p_j$ for all $k<j\leq n$, by the assumptions on $k$. Therefore, $\bdp \geq_{\sigma} \bda >_{\sigma} \bdb$. Moreover, $\bdp=\bda$ if and only if $\bdq=\bdb$. Thus, $[\bdp,\bdq] \geq_{\sigma} [\bda,\bdb]$, implying that $[\bda,\bdb]$ is standard.

    Conversely, suppose that no such $k$ exists. 
    Since $\bda\ne \bdb$ and $[\bda,\bdb]$ is semi-standard, we have $a_{k'}<b_{k'}$ from \eqref{eqn:cond_1} for $k'\coloneqq \min\{j\in [n]:a_j\ne b_j\}$. Due to the non-existence of such $k$ and the choice of $k'$, we can find $j>k'$ such that $a_j<b_j$. Then, we  consider the tuples $\bdp=(p_1,\ldots, p_{n})$ and $\bdq=(q_1,\ldots, q_{n})$ such that 
    \[
        p_\ell=
        \begin{cases}
            a_{\ell}, & \text{if $\ell\ne j$,} \\
            b_j, & \text{if $\ell=j$,}
        \end{cases}
        \qquad \text{and} \qquad
        q_\ell=
        \begin{cases}
            b_{\ell}, & \text{if $\ell\ne j$,} \\
            a_j, & \text{if $\ell=j$.}
        \end{cases}
    \]
    It is clear that $\bda>_{\sigma}\bdp \ge_{\sigma} \bdq$ and $\supp ([\bdp,\bdq])=\supp([\bda,\bdb])$. 
    This indicates that $[\bda,\bdb]$ is not standard.
\end{proof}

\begin{Proposition}
    \label{StandardExist}
     A semi-standard $p\times n$ tableau $\bfA=[\bda^1,\dots,\bda^p]$ is standard if and only if $[\bda^h, \bda^{k}]$ is standard for every $1 \leq h < k \leq p$.
\end{Proposition}
\begin{proof}
    Let $\bfB=[\bdb^1, \ldots, \bdb^{p}]$ be the tableau constructed in \Cref{algo_standard_tab} from $\bfA$. According to \Cref{lem:algo_works}, $\bfB$ is standard. Furthermore, it follows from the construction of $\bfB$ in \Cref{algo_standard_tab} and \Cref{StandardProp} that $[\bdb^h, \bdb^{k}]$ is standard for every $1 \leq h < k \leq p$.
    \begin{enumerate}[a]
        \item Suppose that $\bfA$ is standard. It follows from the uniqueness of standard tableau that $\bfA=\bfB$. Consequently, $[\bda^h, \bda^{k}]$ is standard for every $1 \leq h < k \leq p$.

        \item Conversely, suppose that $[\bda^h, \bda^{k}]$ is standard for every $1 \leq h < k \leq p$. We prove that $\bfA$ is standard. For this purpose, in the following, we prove by induction on $n$ that $\bda^i=\bdb^i$ for each $i$.  Suppose that the $(i,j)$-entries of the semi-standard tableaux $\bfA$ and $\bfB$ are $a_{i,j}$ and $b_{i,j}$, respectively.

            Notice that the base where $n=1$ is clear, since $a_{1,1}\le a_{2,1}\le \cdots\le a_{p,1}$ and $b_{1,1}\le b_{2,1}\le \cdots \le b_{p,1}$
            from \eqref{eqn:cond_1}.
            Now, consider the general case. For 
            $i\in [p]$ and $j\in [n]$,
            let ${\bda}^i_{j}=({a}_{i,1},\dots,{a}_{i,j})$ and $\bdb^i_{j}=(b_{i,1},\dots,b_{i,j})$, by abuse of notation. It follows from \Cref{StandardProp} that both $[\bda^h_{n-1},\bda^k_{n-1}]$ and $[\bdb^h_{n-1},\bdb^k_{n-1}]$ are standard for $1\le h \le k \le p$. Furthermore, it is clear that $\supp([\bda^1_{n-1},\bda^2_{n-1},\dots,\bda^p_{n-1}])= \supp([\bdb^1_{n-1},\bdb^2_{n-1},\dots,\bdb^{p}_{n-1}])$. Hence, by the induction hypothesis for the $n-1$ case, we know that $\bda^i_{n-1}=\bdb^i_{n-1}$ for each $i$.

            Moreover, for $1\le h< k \le p$, since $[{\bda}^h_{n}={\bda}^h, {\bda}^k_{n}={\bda}^k]$ and $[{\bdb}^h_{n}={\bdb}^h, {\bdb}^k_{n}={\bdb}^k]$ are standard, we have the following observations from \Cref{StandardProp}:
            \begin{itemize}
                \item if $\bda^h_{n-1}=\bdb^h_{n-1}>_{\sigma} \bda^k_{n-1}=\bdb^k_{n-1}$, then ${a}_{h,n}\ge {a}_{k,n}$ and $b_{h,n}\ge b_{k,n}$;
                \item if $\bda^h_{n-1}=\bdb^h_{n-1}=\bda^k_{n-1}=\bdb^k_{n-1}$, then ${a}_{h,n} \le {a}_{k,n}$ and $b_{h,n}\le b_{k,n}$.
            \end{itemize}
            Since $\sigma$ is a total order, we deduce readily from \Cref{obs_inc} that the ordered multiset $\{a_{1,n},a_{2,n},\dots,a_{p,n}\}$ and $\{b_{1,n},b_{2,n},\dots,b_{p,n}\}$ are uniquely determined. Since $\supp_{n}(\bfA)=\supp_{n}(\bfB)$, we have $a_{i,n}=b_{i,n}$ for each $i$. In short, $\bda^i_{n}=\bdb^i_{n}$ for each $i$.

            Consequently, $\bfA=\bfB$. In particular, $\bfA$ is standard.
            \qedhere
    \end{enumerate}
\end{proof}

\begin{Corollary}
    \label{cor:sub_standard} 
    Any subtableau obtained from a standard tableau by removing rows or rightmost columns remains standard.
\end{Corollary}

\section{Multi-Rees algebras of standardizable Ferrers diagrams}

Since we have gathered all the necessary tools, we establish the Gr\"{o}bner bases of the presentation ideals of multi-Rees algebras and their special fibers in this section. 

Let $\calD$ be an $n$-dimensional Ferrers diagram as in \Cref{nF}. We begin by defining a monomial order on the base rings $R=\KK[\bdx]= \KK[x_{i,j} : i\in [n], j\in [m]]$, $T= \KK[\bdT]= \KK[T_{\bda}:\bda\in \calD$, and $S=R[\bdT]=R[T_\bda:\bda\in \calD]$.

\begin{Definition}
    \label{order}
    \begin{enumerate}[a]
        \item The variables in $R$ are naturally ordered: $x_{i,j}>x_{i',j'}$ if and only if $i<i'$ or $i=i'$ and $j<j'$. We apply the lexicographic order to $R$ based on this natural order.

        \item For any vectors $\bda$ and $\bdb$ in $\calD $, we define $T_{\bda} > T_{\bdb}$ if $\bda >_{\sigma} \bdb$ according to \Cref{orderTuple}. 
            The induced lexicographic order on the polynomial ring $T$ from this order of the indeterminates is also denoted by ${\sigma}$, by abuse of notation.

        \item As in the proof of \Cref{thm:ReesIdeal}, let $\sigma'$ denote the product order of the lexicographic order $R$, with the order $\sigma$ on $T$.

        \item When working with a  $p\times n$  semi-standard tableau $\bfA=[\bda^1,\dots,\bda^p]$ with row vectors in $\calD$, the monomial $T_{\bda^1}T_{\bda^2}\cdots T_{\bda^p}$ in the polynomial ring $T$ is denoted by $\bdT_\bfA$. 
    \end{enumerate}
\end{Definition}

\begin{Remark}
    If two $p\times n$ semi-standard tableaux, $\bfA$ and $\bfB$, have row vectors in $\calD$, then $\bdT_{\bfA}>_{\sigma}\bdT_{\bfB}$ if and only if $\bfA>_{\sigma}\bfB$. This implies that if $\supp(\bfA)=\supp(\bfB)$ and $\bfB$ is standard, then the leading term of $\bdT_\bfA-\bdT_\bfB$ is $\bdT_\bfA$.
\end{Remark}

\begin{Notation}
    \label{tildeD}
    Let $\calD$ be an $n$-dimensional Ferrers diagram and $r$ be a positive integer. By the loose restriction on the integer $m$ in \Cref{nF}, without loss of generality, we may assume that $r\le m$. We will denote the Cartesian product $\calD\times [r]$ by $\widetilde{\calD_r}$. It is clear that this is an $(n+1)$-dimensional Ferrers diagram. Furthermore, if $R_{\calD}\coloneqq \KK[x_{i,j} : i\in [n], j\in [m]]$ is the base ring for considering the Ferrers ideal $I_{\calD}$, then we may choose $R_{\widetilde{\calD_r}}\coloneqq R_{\calD}[x_{n+1,1},\dots,x_{n+1,m}]$ for considering the Ferrers ideal $I_{\widetilde{\calD_r}}$. 
\end{Notation}

\begin{Remark}
    \label{extendtoN+1}
    We have a natural isomorphism $\calF(\bigoplus_{i=1}^{r}I_{ \calD}) \cong \calF(I_{ \widetilde{\calD_r}})$ of special fiber rings, which is clear from the
    equality of the ideal $I_{ \widetilde{\calD_r}}$ with the product ideal $(x_{n+1,1},\dots,x_{n+1,r})I_{\calD}$ in $R_{\widetilde{\calD_r}}$.
\end{Remark}

The proposition below is analogous to Proposition 2.11 in \cite{DeNegri}.

\begin{Proposition}
    \label{FiberofOneGen}
    Let $\calD$ be a rectangular $n$-dimensional Ferrers diagram,  $r$ be a positive integer, and $\widetilde{\calD_r}$ be the $(n+1)$-dimensional Ferrers diagram associated with $\calD$. Then the set $\calG\coloneqq \bfI(\widetilde{\calD_r})$ defined in \Cref{2-minors} is a Gr\"{o}bner basis of the presentation ideal $\calK$ of the special fiber $\calF(\bigoplus_{i=1}^{r}I_{\calD})$ with respect to the monomial order $\sigma$ defined in \Cref{order}.
\end{Proposition}

\begin{proof}
    By \Cref{extendtoN+1}, we need to show  $\calG\coloneqq \bfI(\widetilde{\calD_r})$ is a Gr\"{o}bner basis of the presentation ideal $\calK$ of the special fiber ring $\calF(I_{\widetilde{\calD_r}})$. By \cite[Lemma 4.1]{Sturmfels}, the set
    \begin{align*}
        \calG^\dag\coloneqq \{\bdT_{\bfA}-\bdT_{\bfB} &: \text{ $\bfA$ and $\bfB$ are two semi-standard tableaux with}\\
        & \quad \text{ row vectors in $\widetilde{\calD_r}$ and $\supp(\bfA)=\supp(\bfB)$}\}
    \end{align*}
    is a generating set of the presentation ideal of $\calF(I_{\widetilde{\calD_r}})$ as a $\KK$-vector space. Since $\calG\subseteq \calG^\dag$, it suffices to show that $\ini_\sigma(f)\in (\ini_\sigma(g):g\in \calG)$ for every $f\in\calG^\dag$. Without loss of generality, suppose that $f=\bdT_{\bfA}-\bdT_{\bfB}$ with $\text{supp}(\bfA)=\text{supp}(\bfB)$ and $\ini_{\sigma }(f)=\bdT_{\bfA}$. Furthermore, assume that $\bfA=[\bda^1,\dots,\bda^p]$.
    \begin{enumerate}[a]
        \item If $\bfA$ is not standard, by \Cref{StandardExist}, there exist $1\le h<k\le p$ such that the $2$-row semi-standard tableau $[\bda^h,\bda^k]$ is not standard. Applying \Cref{algo_standard_tab}, we can find a standard tableau $[\bdp,\bdq]$ from $[\bda^h,\bda^k]$. Since $\widetilde{\calD_r}$ is also a rectangular diagram, $\bdp,\bdq\in \widetilde{\calD_r}$. Whence, the binomial $f'\coloneqq T_{\bda^h}T_{\bda^k}-T_{\bdp}T_{\bdq}$ belongs to $\calG$. Notice that $\ini_\sigma(f')=T_{\bda^h}T_{\bda^k}$, which divides $\ini_\sigma(f)=\bdT_{\bfA}$.
            Therefore, $\ini_\sigma(f) \in (\ini_\sigma(g):g\in \calG)$.
        \item If $\bfA$ is standard, then $\bfB \geq_\sigma \bfA$. Since $\ini_\sigma(f)=T_\bfA$, this implies that $T_\bfB=T_\bfA$ and $f=0$. This case is trivial. \qedhere
    \end{enumerate}
\end{proof}

The previous proof relies heavily on the existence of the standard tableau in the case of rectangular diagrams. The following lemma demonstrates that this existence can be extended to a standardizable Ferrers diagram.

\begin{Lemma}
    \label{LexClosed}
    Let $\calD$ be an $n$-dimensional Ferrers diagram that is standardizable.
    \begin{enumerate}[a]
        \item  Let $[\bda,\bdb]$ and $[\bdp,\bdq]$ be two distinct $2\times n$ semi-standard tableaux such that $\supp ([\bdp,\bdq])=\supp([\bda,\bdb])$. If $[\bdp, \bdq]$ is standard and $\bda,\bdb\in \calD$, then $\bdp,\bdq\in \calD$.
        \item  Let $[\tilde{\bda},\tilde{\bdb}]$ and $[\tilde{\bdp},\tilde{\bdq}]$ be two distinct $2\times (n+1)$ semi-standard tableaux such that $\supp ([\tilde{\bdp},\tilde{\bdq}])=\supp([\tilde{\bda},\tilde{\bdb}])$. If $[\tilde{\bdp}, \tilde{\bdq}]$ is standard and $\tilde{\bda},\tilde{\bdb}\in \widetilde{\calD_r}$, then $\tilde{\bdp},\tilde{\bdq}\in \widetilde{\calD_r}$.
    \end{enumerate}
\end{Lemma}
\begin{proof}
    \begin{enumerate}[a]
        \item By the uniqueness of standard tableau, we know $[\bda,\bdb]$ is not standard.  In particular, $\bda\ne \bdb$. But since $[\bda,\bdb]$ is semi-standard, there is a $k$ such that $a_i=b_i$ for $i<k$, and $a_k<b_k$. This implies that $a_i=p_i=q_i=b_i$ for $i<k$ and $a_k=p_k<q_k=b_k$. Since $[\bdp,\bdq]$ is standard, by \Cref{StandardProp}, we must have $p_j=\max\{a_j,b_j\} \ge q_j=\min\{a_j,b_j\}$ for $j>k$. 
            Since $\calD$ is standardizable and $\bda,\bdb\in \calD$, we obtain $\bdp\in \calD$ by \Cref{def:SDP}. Notice that $q_j\leq b_j$ for each $j$.
            Since $\bdb$ is an element of $\calD$, we can conclude that $\bdq$ is also an element of $\calD$ due to the Ferrers property of $\calD$.

        \item 
            Since $\calD$ is standardizable, it is clear that $\widetilde{\calD_r}$
            is also standardizable. Thus, it remains to apply the previous part with respect to $\widetilde{\calD_r}$.
            \qedhere
    \end{enumerate}
\end{proof}

We are now ready to state one of the main theorems of this work.

\begin{Theorem}
    \label{FiberD}
    Let $\calD$ be a standardizable $n$-dimensional Ferrers diagram, $I_{\calD}$ be the ideal associated with $\calD$, and $r$ be a positive integer. Then the set $\calG\coloneqq \bfI(\widetilde{\calD_r})$ defined in \Cref{2-minors} is a Gr\"{o}bner basis of the presentation ideal $\calK$ of the special fiber $\calF(\bigoplus_{i=1}^{r}I_{ \calD})$. In particular, $\calF(\bigoplus_{i=1}^{r} I_{\calD})$ is Koszul.
\end{Theorem}
\begin{proof}
    By \Cref{extendtoN+1}, \Cref{LexClosed}, \cite[Proposition 4.13]{Sturmfels}, and \cite[Proposition 2.1]{DeNegri}, it suffices to assume that $\calD$ is a rectangular Ferrers diagram. Meanwhile, the corresponding result for this special case has already been given in \Cref{FiberofOneGen}.
\end{proof}

\begin{Corollary}
    \label{CMFiber}
    Let $\calD$ be a standardizable $n$-dimensional Ferrers diagram and $I_{\calD}$ be the ideal associated with $\calD$. Then the special fiber $\calF(\bigoplus_{i=1}^{r} I_{\calD})$ is a Cohen--Macaulay normal domain. 
    Furthermore, if the field $\KK$ has characteristic zero, $\calF(\bigoplus_{i=1}^{r} I_{\calD})$ has rational singularities. If $\KK$ has positive characteristic, $\calF(\bigoplus_{i=1}^{r} I_{\calD})$ is $F$-rational .
\end{Corollary}

\begin{proof}
    Notice that $\calF(\bigoplus_{i=1}^{r} I_{\calD})$ is isomorphic to the semigroup ring $\KK[I_{\widetilde{\calD_r}}]$.
    By \Cref{FiberD} and \Cref{2-minors}, the presentation ideal $\calK$ of $\calF(\bigoplus_{i=1}^{r} I_{\calD})$ has a squarefree initial ideal. Therefore, $\calF(\bigoplus_{i=1}^{r} I_{\calD})$ is normal by \cite[Proposition 13.15]{Sturmfels}, and it is Cohen--Macaulay by
    \cite[Theorem 1]{Hochster}.

    The following argument from \cite[Corollary 2.3]{Sagbi} proves the statements regarding the singularities. We cite it here for readers' convenience.
    By Hochster \cite[Proposition 1]{Hochster}, since $\calF(\bigoplus_{i=1}^{r} I_{\calD})$ is normal, it is a direct summand of a polynomial ring. Thus by a theorem of Boutot \cite{Boutot}, $\calF(\bigoplus_{i=1}^{r} I_{\calD})$ has rational singularities if the characteristic of $\KK$ is zero.
    Furthermore, by a theorem of Hochster and Huneke \cite{MR1044348}, $\calF(\bigoplus_{i=1}^{r} I_{\calD})$ is strongly $F$-regular. In particular, it is $F$-rational, if $\KK$ has positive characteristic.
\end{proof}

In the last part of this work, we investigate the multi-Rees algebras. The following lemma shows that
any collection of Ferrers ideals 
satisfies the strong $\ell$-exchange property as defined in \Cref{def:Strong_L_Ex_P}. It paves the way for finding the presentation equations of the multi-Rees algebra $\calR(\bigoplus_{i=1}^{r} I_{\calD})$.

\begin{Lemma}
    \label{LEx}
    Let $\calD_1,\dots,\calD_r$ be a collection of Ferrers diagrams, where $\calD_k$ is $n_k$-dimensional for each $k$. Set $n\coloneqq \max\{n_1,\dots,n_r\}$. Let $m$ be large enough so that we have the Ferrers ideal $I_{\calD_k} \subset R_k\coloneqq \KK[x_{i,j},i\in [n_k],j\in [m]] \subseteq  R \coloneqq \KK[x_{1,1},\dots,x_{1,m},x_{2,1},\dots,x_{2,m},\dots,x_{n,1},\dots,x_{n,m}]$, for each $k$. Then, the collection of ideals $I_{\calD_1},\dots,I_{\calD_r}$ satisfies the strong $\ell$-exchange property.
\end{Lemma}
\begin{proof}
    For notational simplicity, we denote the variables 
    \[
        x_{1,1},\dots,x_{1,m},x_{2,1},\dots,x_{2,m},\dots,x_{n,1},\dots,x_{n,m} 
    \]
    in $R$ by $z_1, z_2, \dots,z_{nm}$.  Suppose that $\{f_{k,1},\dots,f_{k,\mu_i}\}$ is a minimal monomial generating set of $I_{\calD_k}$, for $k\in [r]$. Furthermore, suppose that $(w_1,\dots,w_r)\in \NN^r$, and $u=\prod_{k=1}^r\prod_{\ell=1}^{w_k} f_{k,\alpha_{k,\ell}}$, $v=\prod_{k=1}^r \prod_{\ell=1}^{w_k} f_{k,\beta_{k,\ell}}$ satisfy
    \begin{enumerate}[i]
        \item $\deg_{z_t}(u)=\deg_{z_t}(v)$ for $t=1,\dots,q-1$ with $q\le nm-1$,
        \item $\deg_{z_q}(u)<\deg_{z_q}(v)$.
    \end{enumerate}
    Suppose that $z_q=x_{i_0,j_0}$ in the above notation. 
    Then,
    \[
        \sum_{j=1}^{j_0-1} \deg_{x_{i_0,j}}(u) =\sum_{j=1}^{j_0-1} \deg_{x_{i_0,j}}(v)    
        \quad\text{ and }\quad
        \sum_{j=1}^{j_0} \deg_{x_{i_0,j}}(u) < \sum_{j=1}^{j_0} \deg_{x_{i_0,j}}(v).    
    \]
    Since
    \[
        \sum_{j=1}^m \deg_{x_{i_0,j}}(u)
        = \sum_{\substack{k\in [r],\\ n_k\ge i_0}} w_k
        =\sum_{j=1}^m \deg_{x_{i_0,j}}(v),
    \]
    we must have $j_0\le m-1$.
    Furthermore, we can find $j_1$ with $j_0<j_1\le m$ such that $\deg_{x_{i_0,j_1}}(u)\ge 1$. Let $z_{q'}=x_{i_0,j_1}$. It is clear that $q'>q$.
Whence, we can find $k,\ell$ such that $f_{k,\alpha_{k,\ell}}$ divides $u$ and $\deg_{z_{q'}}(f_{k,\alpha_{k,\ell}})= 1$. Since ${\calD_{k}}$ is a Ferrers diagram, it is clear that $z_q f_{k,\alpha_{k,\ell}}/z_{q'}\in I_{\calD_{k}}$.
\end{proof}

We conclude this work with the following result on the multi-Rees algebras $\calR(\bigoplus_{i=1}^{r} I_{\calD})$. As announced, it generalizes the previous \Cref{thm:3D}.

\begin{Theorem}
    \label{ReesnD}
    Let $\calD$ be a standardizable $n$-dimensional Ferrers diagram, $I_{\calD}$ be the ideal associated with $\calD$, and $r$ be a positive integer.  Furthermore, let 
    \[
        \calH\coloneqq \{x_{k,k_1}T_{(\bda,i)}-x_{k,k_2}T_{(\bdb,i)}: i\in [r], \bda,\bdb\in \calD, \text{ and } x_{k,k_1}\bdx_{\bda}=x_{k,k_2}\bdx_\bdb \}.
    \]
    Then, $\calG'=\bfI(\widetilde{\calD_r}) \cup \calH$ is a Gr\"{o}bner basis of the presentation ideal $\calJ$ of the multi-Rees algebra $\calR(\bigoplus_{i=1}^{r}I_{ \calD})$.
    In particular, $\calR(\bigoplus_{i=1}^{r} I_{\calD})$ is a Koszul Cohen--Macaulay normal domain.
    Furthermore, if the field $\KK$ has characteristic zero, $\calR(\bigoplus_{i=1}^{r} I_{\calD})$ has rational singularities. If $\KK$ has positive characteristic, $\calR(\bigoplus_{i=1}^{r} I_{\calD})$ is $F$-rational.
\end{Theorem}

\begin{proof}
    The fact that $\mathcal{G}'$ constitutes a Gr\"obner basis for the presentation ideal of $\mathcal{R}(\bigoplus_{i=1}^{r}I_{\mathcal{D}})$ with respect to the order $\sigma'$ defined in \Cref{order} is a direct outcome of \Cref{FiberD}, \Cref{LEx}, and \Cref{thm:ReesIdeal}. The remaining assertions can be established through a comparable argument employed in the proof of \Cref{CMFiber}.
\end{proof}

We conclude this work with a simple example that demonstrates that the same proving strategy cannot be applied to the case of multi-Rees algebras associated with distinct Ferrers diagrams. A new order must be found to apply similar proving steps. 

\begin{Example}
    \label{exam:diffD}    
    Let $\calD_1$ be the minimal $2$-dimensional Ferrers diagram that contains the lattice points $(4,1)$ and $(2,4)$. Similarly, let $\calD_2$ be the similar diagram that contains the lattice points  $(1,3)$ and $(3,1)$. They are standardizable.   In the reduced Gr\"obner basis of the presentation ideal $\calK$ of the special fiber ring $\calF(I_{\calD_1} \oplus I_{\calD_2})$ with respect to the natural lexicographic order, there exist $4$ extra binomials of degree $3$.
\end{Example}

\begin{acknowledgment*}
    The second author is partially supported by the ``the Fundamental Research Funds for Central Universities'' and ``the Innovation Program for Quantum Science and Technology'' (2021ZD0302902).
\end{acknowledgment*}

\bibliography{ReesBib}
\end{document}